\documentclass[12pt]{amsart}
\pdfoutput=1

\usepackage{amsmath, amsthm, amssymb}
 \usepackage{ifpdf}
\ifpdf
\usepackage[pdftex]{graphicx}
\else
\usepackage[dvips]{graphicx}
\fi
\usepackage{tikz}
 	 \usetikzlibrary{arrows,backgrounds} 
\usepackage[all]{xy}

\usepackage[pdftex,plainpages=false,hypertexnames=false,pdfpagelabels]{hyperref}
\newcommand{\arxiv}[1]{\href{http://arxiv.org/abs/#1}{\tt arXiv:\nolinkurl{#1}}}

\newcommand{\googlebooks}[1]{(preview at \href{http://books.google.com/books?id=#1}{google books})}
\theoremstyle{plain}
\newtheorem{thm}{Theorem}[]
\newtheorem{cor}[thm]{Corollary}
\newtheorem{lem}[thm]{Lemma}

\newtheorem{prop}[thm]{Proposition}

\theoremstyle{definition}

\newtheorem{rem}[thm]{Remark}

\newcommand{\comment}[1]{}

\newcommand{\be}{\begin{enumerate}}
\newcommand{\ee}{\end{enumerate}}

\input xy
\xyoption{all}

\begin{document}

\title{A rotational approach to triple point obstructions}

\author{Noah Snyder}

%\address{Department of Mathematics, University of California, Berkeley, 94720}

\date{\today}

%%% ----------------------------------------------------------------------

\begin{abstract}
Subfactors where the initial branching point of the principal graph is $3$-valent are subject to strong constraints called triple point obstructions.  Since more complicated initial branches increase the index of the subfactor, triple point obstructions play a key role in the classification of small index subfactors.  There are two strong triple point obstructions called the triple-single obstruction and the quadratic tangles obstruction.  Although these obstructions are very closely related, neither is strictly stronger.  In this paper we give a more general triple-point obstruction which subsumes both.  The techniques are a mix of planar algebraic and connection-theoretic techniques with the key role played by the rotation operator.
\end{abstract}

%%% ----------------------------------------------------------------------
\maketitle
%\tableofcontents
%%% ----------------------------------------------------------------------
%
%%%%%%%%%%%%%%%%%%%%%%%%%%%%%%%%%%%%%%%%%%%%%%%%%%%
%%%%%%%%%%%%%%%%%%%%%%%%%%%%%%%%%%%%%%%%%%%%%%%%%%%
%%%%%%%%%%%%%%%%%%%%%%%%%%%%%%%%%%%%%%%%%%%%%%%%%%%
\section{Introduction}

%%% ----------------------------------------------------------------------

%%%%%%%%%%%%%%%%%%%%%%%%%%%%%%%%%%%%%%%%%%%%%%%%%%%%%%%

The principal graph of a subfactor begins with a type $A$ string and then hits an initial branch point (unless the graph is $A_k$ or $A_\infty$).  It is natural to stratify subfactors based on how complex this initial branch point is.  Furthermore, complex initial branches increase the norm of the graph and thus the index of the subfactor.  This means that small index subfactors can only have simple initial branches.   The simplest possibility is an initial triple point (in this case the dual graph also begins with a triple point).  Subfactors beginning with an initial triple point are subject to strong constraints known as triple point obstructions.  For example, a triple point obstruction due to Ocneanu shows that as long as the index is at least $4$ the initial triple point must be at odd depth.   These triple point obstructions play a crucial role in the classification of small index subfactors \cite{MR1317352,1007.1730, index5-part2, index5-part3, index5-part4}.

The current state of the art of triple point obstructions is given in \cite{index5-part2}, but the status is somewhat unsatisfatory as there are two main results neither of which is strictly stronger than the other.  One result applies more generally and proves a certain inequality, while the other (due to Jones  \cite{math/1007.1158}) has stricter assumptions but replaces the inequality with a finite list of values.  The former is proved using connections and the latter using planar algebras.  The main result of this paper is a mutual generalization of these two triple point obstructions which proves the stronger conclusion using only the weaker assumptions.  As one might expect, this paper uses a mix of connections and planar algebras following \cite{index5-part3}.  Furthermore, one can think of this argument as giving an alternate proof of the triple point obstruction from \cite{math/1007.1158}.

Before stating the three relevant results, we fix some notation which we will use throughout the paper.   Suppose that $N \subset M$ is an $n-1$ supertransitive finite index subfactor, let $\Gamma$ and $\Gamma'$ denote the principal and dual principal graphs.  Let $[k]$ denote the quantum number $(\nu^k - \nu^{-k})/(\nu-\nu^{-1})$ where $\nu$ is a number such that the index is $[2]^2$.  Let $\beta$ and $\beta'$ denote the initial triple points at depth $n-1$ (which is necessarily odd by Ocneanu's obstruction), let $\alpha_1$ and $\gamma_1$ be the vertices at depth $n-2$, let $\alpha_2$ and $\alpha_3$ be the two vertices at depth $n$ on $\Gamma$, and let $\gamma_2$ and $\gamma_3$ be the two vertices at depth $n$ on $\Gamma'$.  We will conflate vertices with the corresponding simple bimodules and the corresponding simple projections in the planar algebra.  Assume without loss of generality that $\dim \alpha_2 \geq \dim \alpha_3$ and $\dim \gamma_2 \geq \dim \gamma_3$.

\begin{thm}[Triple-single obstruction]\cite[Thm 3.5]{index5-part2} \label{TS}
If $\gamma_3$ is $1$-valent, then $$\dim(\alpha_2)-\dim(\alpha_3) \leq 1.$$
\end{thm}

\begin{thm}[Quadratic tangles obstruction]\cite{math/1007.1158} \label{QT}
Suppose that $\gamma_3$ is $1$-valent and that $\gamma_2$ is $3$-valent, then $$r+\frac{1}{r} = \frac{\lambda+\lambda^{-1}+2}{[n][n+2]}+2$$
where $\lambda$ is the scalar by which rotation acts on the $1$-dimensional perpendicular complement of Temperley-Lieb at depth $n$ and $r = \frac{\dim(\alpha_2)}{\dim(\alpha_3)}$.
\end{thm}

Since $\lambda$ is a root of unity, we know that $-2 \leq \lambda + \lambda^{-1} \leq 2$.  Hence the QT obstruction gives an inequality, and (as observed by Zhengwei Liu) this inequality turns out to be precisely the one in the triple-single obstruction \cite[Lemma 3.3]{index5-part2}.  Thus the QT obstruction is stronger (replacing an interval of possibilities with a finite list) when both apply, but the triple-single obstruction has a weaker assumption.  The main result of this paper is the following mutual generalization of Theorems \ref{TS} and \ref{QT}.

\begin{thm} \label{mainresult}
Suppose that $\gamma_3$ is $1$-valent, then 
$$r+\frac{1}{r} = \frac{\lambda+\lambda^{-1}+2}{[n][n+2]}+2.$$
\end{thm}

The main ideas in this paper came out of joint work with Scott Morrison, and I would like to thank him for many helpful conversations.  I would also like to thank Vaughan Jones, Dave Penneys, and Emily Peters.   This work was supported by an NSF Postdoctoral Fellowship at Columbia University and DARPA grant HR0011-11-1-0001.

\section{Background}

We quickly summarize the key idea of \cite[\S 5.2]{index5-part3}, which is that the action of rotation on the planar algebra can be read off from the connection.  Since rotational eigenvalues must be roots of unity this gives highly nontrivial constraints on candidate connections.  We assume that the reader is familiar with both planar algebras and connections, referring the reader to \cite{index5-part3} for more detail.

Given a subfactor $N \subset M$ we get a certain collection of matrices called a connection.  This connection depends on a choice of certain intertwiners, and thus is only well-defined up to gauge automorphisms.  Let the branch matrix $U$ denote the $3$-by-$3$ matrix coming from the connection at the initial branch vertex of $\Gamma$.   The key idea from \cite[\S 5.2]{index5-part3} is that there is a canonical gauge choice for $U$, called the diagrammatic branch matrix, coming from the planar algebra.  This choice is both easy to recognize and has nice properties, as captured by the following two results.

\begin{lem} \cite[Lemma 5.6]{index5-part3} \label{recognize}
When $n$ is odd the diagrammatic branch matrix is characterized within its gauge class by the property that all the entries in the first row and column are positive real numbers.
\end{lem}

\begin{prop}
Let $U$ be the diagrammatic branch matrix for a subfactor with an initial triple point.  Suppose that $x$ is an $n$-box in the perpendicular complement of Temperley-Lieb, and write $x = a_2\frac{\alpha_2}{\sqrt{\dim \alpha_2}} + a_3 \frac{\alpha_3}{\sqrt{\dim \alpha_3}}$.  Let $(c_1, c_2, c_3) = U(0, a_2, a_3)$.  Then $c_1= 0$ and $c_2 \frac{\gamma_2}{\sqrt{\dim \gamma_2}}  + c_3 \frac{\gamma_3}{\sqrt{\dim \gamma_3}}$ is $\rho^{\frac{1}{2}} (x)$.
\end{prop}
\begin{proof}
This is a restatement of \cite[Corollary 5.3]{index5-part3} in our special case.  See \cite[pp. 18--19]{index5-part3} for a worked example.
\end{proof}

In order to apply the previous proposition we will want an explicit formula for vectors in the perpendicular complement to Temperley-Lieb in the $n$-box space and the action of rotation there.  Recall that the rotation $\rho$ preserves shading and thus is an endomorphism of each box space, while $\rho^{\frac{1}{2}}$ changes rotation and thus is a map from one box space to a different box space.  We will use $\lambda$ to denote the scalar by which $\rho$ acts on the $1$-dimensional perpendicular complement to Temperley-Lieb in the $n$-box space.  Note that this is an $n$th root of unity.

\begin{lem}
Let $r = \frac{\dim \alpha_2}{\dim \alpha_3}$ and $\check{r} = \frac{\dim \gamma_2}{\dim \gamma_3}$.  Then  $T = \frac{1}{\sqrt{r}} \alpha_2 -   \sqrt{r} \alpha_3$ and $\check{T} = \frac{1}{\sqrt{\check{r}}} \gamma_2 - \sqrt{\check{r}}  \gamma_3$ are each in the perpendicular complement of Temperley-Lieb.  

Furthermore $\rho^{\frac{1}{2}}(T)=  \sqrt{\lambda} \check{T}$, where  $\sqrt{\lambda}$ is  some square root of the rotational eigenvalue for the action of rotation on the perpendicular compliment of Temperley-Lieb.
\end{lem}
\begin{proof}
These calculations (with slightly different conventions) were done in an early version of \cite{math/1007.1158}.  Seeing that $T$ and $\check{T}$ are perpendicular to Temperley-Lieb is straightforward (you only need to work out their inner product with two specific Jones-Wenzl projections).  Since half-click rotation preserves Temperley-Lieb and is an isometry, it also preserves the perpendicular complement of Temperley-Lieb.  Thus $\rho^{\frac{1}{2}}(T)$ is some scalar multiple of $\check{T}$.  To work out which scalar multiple this is you compute their norms. This tells you that the square of this scalar is $\lambda$.
\end{proof}

\begin{rem}
There are many square roots in this paper.  Other than $\sqrt{\lambda}$ all square roots are positive square roots of positive numbers.  $\sqrt{\lambda}$ will always be chosen such that the previous lemma works.  In the final statement of the main theorem no $\sqrt{\lambda}$ appears, so this subtlety is not very important.
\end{rem}

Combining the previous two results we have the following concrete statement, which will supply the main ingredient of our proof of Theorem \ref{mainresult}.

\begin{cor} \label{eigenformula}
The diagrammatic branch matrix $U$ sends $$(0, \sqrt{\dim(\alpha_3)}, -\sqrt{\dim(\alpha_2)}) \mapsto \sqrt{\lambda} (0, \sqrt{\dim(\gamma_3)}, -\sqrt{\dim (\gamma_2)}).$$
\end{cor}

\section{Proof of Theorem \ref{mainresult}}

The idea of this argument is that having a $1$-valent vertex allows us to solve for the branch matrix, and thus we can read off the rotational eigenvalue (since the diagrammatic branch matrix acts on the appropriate vectors by rotation).  This gives an identity between the dimensions of objects and the rotational eigenvalue.

We begin with a quick calculation of the branch matrix following the proof of the triple-single obstruction \cite[Thm 3.1]{index5-part2}.  Since $\alpha_1$, $\gamma_1$, $\beta$, and $\beta'$ are in the initial string there dimensions are $[n-1]$, $[n-1]$, $[n]$, and $[n]$, respectively.  Since $\gamma_3$ is $1$-valent, we have $\dim \gamma_2 = \frac{[n+2]}{[2]}$ and $\dim \gamma_3 = \frac{[n]}{[2]}$.  Using the $1$-valence of $\gamma_3$ the normalization condition on connections determines the magnitude of several of the entries in the branch matrix.  Furthermore, unitarity of $U$ allows us to work out several more of the entries.  In particular, the branch matrix is gauge equivalent to the matrix below, where $p = \dim(\alpha_2)$ and $q = \dim(\alpha_3)$, where $\sigma$ and $\tau$ are unknown phases, and where $?$ denotes unknown entries which will play no role in the calculation.

$$U = \begin{pmatrix}
\frac{1}{[n]} & \sqrt{[n-1] p} & \sqrt{[n-1] q} \\
\sqrt{\frac{[n-1]}{[2][n]}} & \sigma \sqrt{\frac{p}{[2][n]}} & \tau \sqrt{\frac{q}{[2][n]}} \\
\sqrt{\frac{[n-1][n+2]}{[2][n]^2}} & ? & ? 
\end{pmatrix}$$

The first row and column of this matrix are clearly positive, so by Lemma \ref{recognize} we see that $U$ is the diagrammatic branch matrix.  

\begin{rem}
This matrix is the transpose of the matrix found in \cite{index5-part2} because the calculation there is done for $\Gamma'$ instead of $\Gamma$.  As shown in \cite{index5-part3}, the diagrammatic branch matrices of $\Gamma$ and $\Gamma'$ are always transposes.
\end{rem}

We would like to solve for $\sigma$ and $\tau$.  Orthogonality of the first two rows of $U$ tells us that $$1+\sigma p + \tau q = 0.$$    Although $1+\sigma p + \tau q = 0$ is one equation in two unknowns, it actually determines $\sigma$ and $\tau$ since they are phases: 
$$\sigma = - \frac{1+\tau q}{p}$$
$$1 = \sigma \bar{\sigma} = \frac{1+\tau q}{p} \frac{1+\bar{\tau} q}{p} = \frac{1 + (\tau + \bar{\tau}) q + q^2}{p^2}$$
$$ \tau + \bar{\tau} = \frac{p^2 - q^2 -1}{q}.$$

This determines the real part of $\tau$, and thus $\tau$ itself.   Similarly, $\sigma + \bar{\sigma} = \frac{q^2 - p^2 -1}{p}$.

Now that we have a very explicit understanding of $U$ we apply it to a rotational eigenvector. Corollary \ref{eigenformula}
 tells us that $U$ sends $$(0, \sqrt{q}, -\sqrt{p}) \mapsto \sqrt{\lambda} \left(0, \sqrt{\frac{[n+2]}{[2]}}, -\sqrt{\frac{[n]}{[2]}}\right).$$ 
%(See pages 18--19, note that this uses the calculation from the original QT which relates $\rho^\frac{1}{2}(S)$ and $\check{S}$.)
Looking at the middle coordinate of that identity we see that $$\sigma - \tau = \sqrt{\lambda} \sqrt{\frac{[n+2][n]}{p q}} .$$

Comparing the real parts of both sides yields the following formula
\begin{align*}
(\sqrt{\lambda}+\frac{1}{\sqrt{\lambda}}) \sqrt{\frac{[n+2][n]}{p q}} &= (\sigma + \bar{\sigma}) - (\tau + \bar{\tau})  \\
&= \frac{q^2 - p^2 -1}{p}-\frac{p^2 - q^2 -1}{q} = \frac{(q-p)((p+q)^2 -1)}{pq} \\
&= \frac{(q-p)([n+1]^2 -1)}{pq} =\frac{(q-p)([n][n+2])}{pq}.
\end{align*}

Squaring both sides and rearranging proves the theorem.

\begin{rem}
You might guess that $\sigma - \tau = \sqrt{\lambda} \left(\sqrt{\frac{[n+2][n]}{p q}} \right)$ would give a second condition coming from the imaginary parts.  In fact there's no new information there, because the two sides automatically have the same norm.
\end{rem}

\newcommand{\urlprefix}{}

\bibliographystyle{alpha}
%Included for winedt:
%input "bibliography/bibliography.bib"
\bibliography{bibliography}

\newcommand{\noopsort}[1]{}\def\cprime{$'$} \def\cprime{$'$} \def\cprime{$'$}
\begin{thebibliography}{MPPS10}

\bibitem[Haa94]{MR1317352}
Uffe Haagerup.
\newblock Principal graphs of subfactors in the index range
  {$4<[M:N]<3+\sqrt2$}.
\newblock In {\em Subfactors ({K}yuzeso, 1993)}, pages 1--38. World Sci. Publ.,
  River Edge, NJ, 1994.
\newblock \mathscinet{MR1317352}.

\bibitem[IJMS11]{index5-part3}
M.~{Izumi}, V.~F.~R. {Jones}, S.~{Morrison}, and N.~{Snyder}.
\newblock {Subfactors of index less than 5, part 3: quadruple points}.
\newblock {\em Communications in Mathematical Physics}, September 2011.
\newblock \arxiv{1109.3190}.

\bibitem[Jon03]{math/1007.1158}
Vaughan F.~R. Jones.
\newblock Quadratic tangles in planar algebras, 2003.
\newblock \arxiv{1007.1158}.

\bibitem[MPPS10]{index5-part2}
Scott Morrison, David Penneys, Emily Peters, and Noah Snyder.
\newblock Classification of subfactors of index less than 5, part 2: triple
  points.
\newblock {\em International Journal of Mathematics}, 2010.
\newblock \arxiv{1007.2240}, accepted June 28 2011.

\bibitem[MS10]{1007.1730}
Scott Morrison and Noah Snyder.
\newblock Subfactors of index less than 5, part 1: the principal graph
  odometer.
\newblock {\em Communications in Mathematical Physics}, 2010.
\newblock \arxiv{1007.1730}, accepted June 28 2011.

\bibitem[PT10]{index5-part4}
David Penneys and James Tener.
\newblock Classification of subfactors of index less than 5, part 4:
  cyclotomicity.
\newblock {\em International Journal of Mathematics}, 2010.
\newblock \arxiv{1010.3797}, accepted June 28 2011.

\end{thebibliography}

\end{document}